\documentclass[12pt]{article}
\usepackage{amsmath,amssymb}

\newtheorem{theorem}{Theorem}[section]

\newtheorem{definition}[theorem]{Definition}
\newtheorem{example}[theorem]{Example}
\newtheorem{remark}[theorem]{Remark}

\newenvironment{proof}{\smallskip\par{\sc Proof.}\enspace}%
 {{\unskip\nobreak\hfil\penalty50\hskip2em
          \hbox{}\nobreak\hfil{\rule[-1pt]{5pt}{10pt}}
          \parfillskip=0pt\finalhyphendemerits=0
          \par\medskip}} 

\makeatletter
\def\section{\@startsection {section}{1}{\z@}{3.25ex plus 1ex minus
 .2ex}{1.5ex plus .2ex}{\large\bf}}
\def\subsection{\@startsection{subsection}{2}{\z@}{3.25ex plus 1ex minus
 .2ex}{1.5ex plus .2ex}{\normalsize\bf}}
\@addtoreset{equation}{section} 
\makeatother

\begin{document}

\begin{center}
\LARGE\textsf{A new approach to Poincar\'e-type inequalities on the
Wiener space}
\end{center}

\vspace*{.3in}

\begin{center}
\sc
\large{Alberto Lanconelli}\\
\end{center}

\begin{center}
\sc
Dipartimento di Matematica \\
Universita' degli Studi di Bari\\
 Via E. Orabona, 4\\
 70125 Bari - Italy\\
E-Mail: {\it alberto.lanconelli@uniba.it}
\end{center}

\vspace*{.3in}

\begin{abstract}
We prove a new type of Poincar\'e inequality on abstract Wiener
spaces for a family of probability measures which are absolutely
continuous with respect to the reference Gaussian measure. This
class of probability measures is characterized by the strong
positivity (a notion introduced by Nualart and Zakai in \cite{NZ})
of their Radon-Nikodym densities. Measures of this type do not
belong in general to the class of log-concave measures, which are a
wide class of measures satisfying the Poincar\'e inequality
(Brascamp and Lieb \cite{BL}). Our approach is based on a point-wise
identity relating Wick and ordinary products and on the notion of
strong positivity which is connected to the non negativity of Wick
powers. Our technique leads also to a partial generalization of the
Houdr\'e and Kagan \cite{HK} and Houdr\'e  and P\'erez-Abreu
\cite{HPA} Poincar\'e-type inequalities.
\end{abstract}

\bigskip
\noindent \textbf{Keywords:} Poincar\'e inequality, abstract Wiener
space, Wick product, stochastic exponentials.

\bigskip
\noindent\textbf{Mathematics Subject Classification (2000):} 60H07,
60H30, 60H40.

\section{Introduction}

The \emph{Poincar\'e} or \emph{spectral gap} inequality for the
standard $n$-dimensional Gaussian measure states that
\begin{eqnarray}\label{poincare}
\int_{\mathbb{R}^n}f^2(x)d\mu(x)-\Big(\int_{\mathbb{R}^n}f(x)d\mu(x)\Big)^2\leq
\int_{\mathbb{R}^n}|\nabla f(x)|^2d\mu(x)
\end{eqnarray}
where $f$ is a smooth function,
$d\mu(x)=(2\pi)^{-\frac{n}{2}}\exp\{-\frac{|x|^2}{2}\}dx$ and
$\nabla$ denotes the gradient operator. Inequality (\ref{poincare})
was first
proved by Nash in \cite{Nash} and later on rediscovered by Chernoff in \cite{Chernoff}.\\
The literature concerning extensions, refinements and applications
of the Poincar\'e inequality is very rich. One of the key features
of inequality (\ref{poincare}) is that it is dimension independent
and in fact one can prove (see Gross \cite{G} also for the
connection with logarithmic Sobolev inequalities) its
validity on abstract Wiener spaces (see Section 2 below).\\
An important refinement to inequality (\ref{poincare}) was proposed
by Houdr\'e and Kagan in \cite{HK} where they proved the following
inequalities
\begin{eqnarray}\label{HK introduction}
\sum_{l=1}^{2k}\frac{(-1)^{l+1}}{l!} \int_{\mathbb{R}^n}|
\nabla^lf(x)|^2d\mu(x)&\leq&
\int_{\mathbb{R}^n}f^2(x)d\mu(x)-\Big(\int_{\mathbb{R}^n}f(x)d\mu(x)\Big)^2\nonumber\\
&\leq& \sum_{l=1}^{2k-1}\frac{(-1)^{l+1}}{l!}\int_{\mathbb{R}^n}|
\nabla^lf(x)|^2d\mu(x).
\end{eqnarray}
Here $\nabla^l$ stands for iterated gradients and $|\cdot |$ are the
Euclidean norms on the corresponding $\mathbb{R}^{n^l}$ spaces.
Later on this result was extended to the classical Wiener space by
Houdr\'e and P\'erez-Abreu in \cite{HPA}.\\
One of the most celebrated generalizations of the Poincar\'e
inequality (\ref{poincare}) is due to Brascamp and Lieb \cite{BL}
which extended the validity of (\ref{poincare}) to the class of
log-concave measures. More precisely they proved that if $\nu$ is a
probability measure on $\mathbb{R}^n$ of the form
$d\nu(x)=e^{-V(x)}dx$, for some smooth strictly convex function $V$,
then
\begin{eqnarray}\label{BL introduction}
\int_{\mathbb{R}^n}f^2(x)d\nu(x)-\Big(\int_{\mathbb{R}^n}f(x)d\nu(x)\Big)^2\leq
\int_{\mathbb{R}^n}\langle (\mathcal{H}V(x))^{-1}\nabla f(x),\nabla
f(x)\rangle d\nu(x),
\end{eqnarray}
where $\mathcal{H}V$ is the Hessian matrix of $V$. The previous
inequality was adapted to the context of abstract Wiener spaces by
Feyel and \"Ust\"unel in \cite{FU}.\\
The aim of the present paper is to prove a new type of Poincar\'e
inequality for a class of probability measures on an abstract Wiener
space which is not included in the family of log-concave measures.
We are concerned with probability measures that are absolutely
continuous with respect to the reference Gaussian measure and with
Radon-Nikodym densities that satisfy a strong positivity requirement
(see Section 3 below). This strong positivity condition is the
crucial ingredient of our approach; the notion was introduced in
Nualart and Zakai \cite{NZ} and it is connected to the non
negativity of Wick powers (see Theorem \ref{characterization theorem
for strong positivity} below). It is in fact a formula involving the
Wick product to be the source of inspiration for this paper; the
formula in its simplest appearance reads
\begin{eqnarray}\label{idea}
(f\diamond f)(x)=f(x)^2+\sum_{l\geq
1}\frac{(-1)^l}{l!}[f^{(l)}(x)]^2.
\end{eqnarray}
(Here $f^{(l)}$ denotes the $l$-th derivative of $f$). Identity
(\ref{idea}) appeared for the first time in Nualart and Zakai
\cite{NZ} and independently in Hu and {\O}ksendal \cite{HO}, both
times without proof; see Hu and Yan \cite{HY} and Lanconelli
\cite{L} for proofs in different settings. \\
Taking the integral of
both sides of (\ref{idea}) with respect to the one dimensional
Gaussian measure and using the identity
\begin{eqnarray*}
\int_{\mathbb{R}}(f\diamond
f)(x)d\mu(x)=\Big(\int_{\mathbb{R}}f(x)d\mu(x)\Big)^2
\end{eqnarray*}
one gets
\begin{eqnarray*}
\Big(\int_{\mathbb{R}}f(x)d\mu(x)\Big)^2=\int_{\mathbb{R}}f(x)^2d\mu
(x)+\sum_{l\geq
1}\frac{(-1)^l}{l!}\int_{\mathbb{R}}[f^{(l)}(x)]^2d\mu(x),
\end{eqnarray*}
which coincides with (\ref{HK introduction}) if one formally let $k$
going to infinity.\\
The paper is organized as follows: In Section 2 we introduce the
notation and background material needed to develop our approach to
the Poincar\'e inequality ; Section 3 is devoted to the key notion
of strong positivity with several illustrating examples; finally in
Section 4 we state and prove the main result of the paper (Theorem
\ref{main theorem}) together with a refinement (Theorem \ref{main
theorem 2}) in the spirit of inequality (\ref{HK introduction})

\section{Framework}

The aim of this section is to collect the necessary background
material and fix the notation. For the sake of clarity the topics
will not be treated in their greatest generality but they will be
developed in relation to the application we have in mind, the
Poincar\'e inequality. For more details the interested reader is
referred to the books of Nualart \cite{Nualart}, Janson \cite{J} and
to the paper by Potthoff and Timpel \cite{PT} (the latter reference
is suggested, among other things, for the theory of the spaces
$\mathcal{G}_{\lambda}$ and the notion of Wick product).

\subsection{The spaces $\mathbb{D}^{k,p}$ and $\mathcal{G}_{\lambda}$}

Let $(H,W,\mu)$ be an \emph{abstract Wiener space}, that means
$(H,\langle\cdot,\cdot\rangle_H)$ is a separable Hilbert space which
is continuously and densely embedded in the Banach space
$(W,\Vert\cdot\Vert_W)$ and $\mu$ is a Gaussian probability measure
on the Borel sets $\mathcal{B}(W)$ of $W$ such that
\begin{eqnarray}\label{gaussian characteristic}
\int_{W}e^{i\langle w,w^*\rangle}d\mu(w)=e^{-\frac{1}{2}\Vert
w^*\Vert_H^2},\quad\mbox{ for all }w^*\in W^*.
\end{eqnarray}
Here $W^*\subset H$ denotes the dual space of $W$ (which in turn is
dense in $H$) and $\langle\cdot,\cdot\rangle$ stands for the dual
pairing between $W$ and $W^*$. We will refer to $H$ as the
\emph{Cameron-Martin} space of $W$. Set for $p\geq 1$
\begin{eqnarray*}
\mathcal{L}^p(W,\mu):=\Big\{F:W\to\mathbb{R}\mbox{ such that }\Vert
F\Vert_p:=\Big(\int_W|F(w)|^pd\mu(w)\Big)^{\frac{1}{p}}<+\infty\Big\}.
\end{eqnarray*}
To ease the notation we will write $E[F]$ for $\int_WF(w)d\mu(w)$
and call it the \emph{expectation} of $F$. It follows from
(\ref{gaussian characteristic}) that the map
\begin{eqnarray*}
W^*&\to&\mathcal{L}^2(W,\mu)\\
w^*&\mapsto&\langle w,w^*\rangle
\end{eqnarray*}
is an isometry; we can therefore define for $\mu$-almost all $w\in
W$ the quantity $\langle w,h\rangle$ for $h\in H$  as an element of
$\mathcal{L}^2(W,\mu)$. This element will be denoted
by $\delta(h)$.\\
We now introduce the gradient operator and the class of functions
for which our Poincar\'e inequality will hold true. On the set
\begin{eqnarray*}
\mathcal{S}:=\{F=\varphi(\delta(h_1),...,\delta(h_n))\mbox{ where
}n\in\mathbb{N}, h_1,...,h_n\in H\mbox{ and }\varphi\in
C_0^{\infty}(\mathbb{R}^n)\}
\end{eqnarray*}
define
\begin{eqnarray*}
D(\varphi(\delta(h_1),...,\delta(h_n))):=\sum_{k=0}^n\frac{\partial\varphi}{\partial
x_j}(\delta(h_1),...,\delta(h_n))h_j.
\end{eqnarray*}
The operator $D$ maps $\mathcal{S}$ into $\mathcal{L}^p(W,\mu;H)$;
moreover by means of the integration by parts formula
\begin{eqnarray*}
E[\langle DF,h\rangle_H]=E[F\delta(h)],\quad F\in\mathcal{S}, h\in H
\end{eqnarray*}
one can prove that $D$ is closable in $\mathcal{L}^p(W,\mu)$; we
therefore define the space $\mathbb{D}^{1,p}$ to be closure of
$\mathcal{S}$ under the norm
\begin{eqnarray*}
\Vert F\Vert_{1,p}:=\big(E[|F|^p]+E[\Vert
DF\Vert_H^p]\big)^{\frac{1}{p}}.
\end{eqnarray*}
In a similar way, iterating the definition of $D$ and introducing
for any $k\in\mathbb{N}$ the norms
\begin{eqnarray*}
\Vert F\Vert_{k,p}:=\big(E[|F|^p]+\sum_{j=1}^kE[\Vert
D^jF\Vert_{H^{\otimes j}}^p]\big)^{\frac{1}{p}}.
\end{eqnarray*}
one constructs the spaces $\mathbb{D}^{k,p}$. \\
In order to prove our main results we need to introduce an
additional class of functions. This will be related to the family of
probability measures with respect to which the Poincar\'e inequality
will be proved. To this aim recall that by the Wiener-It\^o chaos
decomposition theorem any element $F$ in $\mathcal{L}^2(W,\mu)$ has
an infinite orthogonal expansion
\begin{eqnarray*}
F=\sum_{n\geq 0} \delta^n(f_n),
\end{eqnarray*}
where $f_n\in H^{\hat{\otimes}n}$, the space of symmetric elements
of $H^{\otimes n}$, and $\delta^n(f_n)$ stands for the multiple
Wiener-It\^o integral of $f_n$. We remark that $\delta^1(f_1)$
coincides with the element $\delta(f_1)$ mentioned above. Moreover
one has
\begin{eqnarray*}
\Vert F\Vert_2^2=\sum_{n\geq 0}n!\Vert f_n\Vert^2_{H^{\otimes n}}.
\end{eqnarray*}
It is useful to observe that if $F$ happens to be in
$\mathbb{D}^{1,2}$ then
\begin{eqnarray*}
E[\Vert DF\Vert_H^2]=\sum_{n\geq 1}n n!\Vert f_n\Vert^2_{H^{\otimes
n}}.
\end{eqnarray*}
For any $\lambda\geq 0$ define the operator $\Gamma(\lambda)$ acting
on $\mathcal{L}^2(W,\mu)$ as
\begin{eqnarray*}
\Gamma(\lambda)\Big(\sum_{n\geq 0}\delta^n(f_n)\Big):=\sum_{n\geq 0}
\lambda^n\delta^n(f_n).
\end{eqnarray*}
If $\lambda\leq 1$ then $\Gamma(\lambda)$ coincides with the
Ornstein-Uhlenbeck semigroup
\begin{eqnarray*}
(P_tf)(w):=\int_Wf(e^{-t}w+\sqrt{1-e^{-2t}}\tilde{w})d\mu(\tilde{w}),\quad
w\in W, t\geq 0
\end{eqnarray*}
(take $\lambda=e^{-t}$) which is a bounded operator. Otherwise
$\Gamma(\lambda)$ is an unbounded operator with domain
\begin{eqnarray*}
\mathcal{G}_{\lambda}:=\Big\{F=\sum_{n\geq 0}
\delta^n(f_n)\in\mathcal{L}^2(W,\mu)\mbox{ such that }\Vert
F\Vert_{\mathcal{G}_{\lambda}}^2:=\sum_{n\geq 0}n!\lambda^{2n}\Vert
f_n\Vert^2_{H^{\otimes n}}<+\infty\Big\}.
\end{eqnarray*}
The family $\{\mathcal{G}_{\lambda}\}_{\lambda\geq 1}$ is a
collection of Hilbert spaces with the property that
\begin{eqnarray*}
\mathcal{G}_{\lambda_2}\subset\mathcal{G}_{\lambda_1}\subset\mathcal{L}^2(W,\mu)
\end{eqnarray*}
for $1<\lambda_1<\lambda_2$. Define
$\mathcal{G}:=\bigcap_{\lambda\geq 1}\mathcal{G}_{\lambda}$ endowed
with the projective limit topology; the space $\mathcal{G}$ turns
out to be a reflexive Fr\'echet space. Its dual $\mathcal{G}^*$ is a
space of generalized functions that can be represented as
$\mathcal{G}^*=\bigcup_{\lambda> 0}\mathcal{G}_{\lambda}$. We remark
that for $F\in\mathcal{L}^2(W,\mu)$ and $G\in\mathcal{G}$ one has
\begin{eqnarray*}
\langle\langle F,G\rangle\rangle=E[F\cdot G].
\end{eqnarray*}
where $\langle\langle\cdot,\cdot\rangle\rangle$ stands for the dual
pairing between $\mathcal{G}^*$ and $\mathcal{G}$. It also follows
from our construction that for $1<\lambda_1<\lambda_2$,
\begin{eqnarray*}
\mathcal{G}\subset\mathcal{G}_{\lambda_2}\subset\mathcal{G}_{\lambda_1}\subset\mathcal{L}^2(W,\mu)
\subset\mathcal{G}_{\frac{1}{\lambda_1}}
\subset\mathcal{G}_{\frac{1}{\lambda_2}}\subset\mathcal{G}^*.
\end{eqnarray*}
Here $\mathcal{G}_{\frac{1}{\lambda_i}}$ represents the dual space
of $\mathcal{G}_{\lambda_i}$, $i=1,2$, respectively.
\begin{remark}
It is not difficult to see that for $\lambda>1$ we have
$\mathcal{G}_{\lambda}\subset \bigcap_{k\geq 1}\mathbb{D}^{k,2}$.
Now, in the context of the classical Wiener space (i.e. $H$ is the
space of absolutely continuous functions on $[0,1]$ with square
integrable derivative and which are zero at zero and $W$ is the
space of continuous functions on $[0,1]$ which are also zero at
zero) the Stroock formula (see \cite{Stroock}) tells that if
$F\in\bigcap_{k\geq 1}\mathbb{D}^{k,2}$ then the elements $f_n$
appearing in its chaos decomposition are explicitly given by
\begin{eqnarray*}
f_n(t_1,...,t_n)=\frac{1}{n!}\int_0^{t_1}\cdot\cdot\cdot\int_0^{t_n}E[D_{s_1,...,s_n}F]ds_1\cdot\cdot\cdot
ds_n
\end{eqnarray*}
where $D_{s_1,...,s_n}F$ stands for the $n$-th Malliavin derivative
of $F$. Using this representation together with the Jensen
inequality we get a sufficient condition for $F$ to be in
$\mathcal{G}_{\lambda}$: \\
Let $F\in\bigcap_{k\geq 1}\mathbb{D}^{k,2}$ be such that the series
\begin{eqnarray*}
\sum_{n\geq 0}\frac{\lambda^{2n}}{n!}E[\Vert D^nF\Vert^2_{H^{\otimes
n}}]
\end{eqnarray*}
converges. Then $F\in\mathcal{G}_{\lambda}$.
\end{remark}
One of the most representative elements of $\mathcal{G}$ is the so
called \emph{stochastic exponential}
\begin{eqnarray*}
\mathcal{E}(h):=\exp\Big\{\delta(h)-\frac{1}{2}\Vert
h\Vert_H^2\Big\},\quad h\in H.
\end{eqnarray*}
(We recall that stochastic exponentials correspond among other
things to Radon-Nikodym derivatives, with respect to the underlying
Gaussian measure $\mu$, of probability measures on
$(W,\mathcal{B}(W))$ obtained through shifted copies of $\mu$ along
Cameron-Martin directions). Its membership to $\mathcal{G}$ can be
easily verified since the Wiener-It\^o chaos decomposition of
$\mathcal{E}(h)$ is obtained
with $f_n=\frac{h^{\otimes n}}{n!}$.\\
Moreover the linear span of the stochastic exponentials, that we
denote with $\mathcal{E}$, is dense in $\mathcal{L}^p(W,\mu)$,
$\mathbb{D}^{k,p}$, for any $p\geq 1$ and $k\in\mathbb{N}$, and $\mathcal{G}$.

\subsection{The Wick product}

For $h,k\in H$ define
\begin{eqnarray*}
\mathcal{E}(h)\diamond\mathcal{E}(k):=\mathcal{E}(h+k).
\end{eqnarray*}
This is called the \emph{Wick product} of $\mathcal{E}(h)$ and
$\mathcal{E}(k)$. Extend this operation by linearity to
$\mathcal{E}$, the linear span of stochastic exponentials, to get a
commutative, associative and distributive (with respect to the sum)
multiplication.\\
Now let $F,G\in\mathcal{L}^2(W,\mu)$ and $F_n,G_n\in\mathcal{E}$
such that
\begin{eqnarray*}
\lim_{n\to +\infty}\Vert F_n-F\Vert_2=0\quad\mbox{ and
}\quad\lim_{n\to +\infty}\Vert G_n-G\Vert_2=0.
\end{eqnarray*}
Set
\begin{eqnarray*}
F\diamond G:=\lim_{n\to +\infty}F_n\diamond G_n.
\end{eqnarray*}
The limit above can not be interpreted in the
$\mathcal{L}^2(W,\mu)$-norm since the Wick product is easily seen to
be an unbounded bilinear operator on that space. The Wick product
$F\diamond G$ of the two square integrable elements $F$ and $G$
lives in the distributional space $\mathcal{G}^*$; more precisely,
\begin{eqnarray}\label{Wick product in G^*}
F,G\in\mathcal{L}^2(W,\mu)\quad\Rightarrow\quad F\diamond
G\in\mathcal{G}_{\frac{1}{\sqrt{2}}}.
\end{eqnarray}
One can also prove that the Wick product is a continuous bilinear
operator on $\mathcal{G}\times\mathcal{G}$ and on
$\mathcal{G}^*\times\mathcal{G}^*$.\\
We mention that for any $F\in\mathbb{D}^{1,2}$ and $h\in H$ one has
$F\diamond\delta(h)\in\mathcal{L}^2(W,\mu)$; moreover the following
identity holds
\begin{eqnarray*}
F\diamond\delta(h)&=&F\cdot\delta(h)-\langle DF,h\rangle_H\\
&=&D_h^*F
\end{eqnarray*}
where $D_h^*$ is the formal adjoint of $\langle D\cdot,h\rangle_H$
in $\mathcal{L}^2(W,\mu)$.\\
A characterizing property of the Wick product is the following: for
any $F,G\in\mathcal{L}^2(W,\mu)$ and $h\in H$,
\begin{eqnarray}\label{charact}
\langle\langle F\diamond
G,\mathcal{E}(h)\rangle\rangle&=&\langle\langle
F,\mathcal{E}(h)\rangle\rangle\cdot\langle\langle
G,\mathcal{E}(h)\rangle\rangle\nonumber\\
&=&E[F\mathcal{E}(h)]\cdot E[G\mathcal{E}(h)]
\end{eqnarray}
In particular for $h=0$ one gets
\begin{eqnarray}\label{Wick expectation}
\langle\langle F\diamond G,1\rangle\rangle=E[F]\cdot E[G].
\end{eqnarray}
We refer the reader to the papers of Da Pelo et al. \cite{DLS} and
\cite{DLS 2} for other properties and deeper results on the Wick
product.

\section{Strongly positive functions}

We now introduce a concept that will play a crucial role in our
approach to the Poincar\'e inequality.
\begin{definition}\label{posivity}
A generalized function $\Phi\in\mathcal{G}^*$ is said to be
\emph{positive} if for any $\varphi\in\mathcal{G}$ one has
$\langle\langle
\Phi,\varphi\rangle\rangle\geq 0$. \\
A generalized function $\Phi\in\mathcal{G}^*$ is said to be
\emph{strongly positive} if for any $\lambda\geq 1$, the generalized
function $\Gamma(\lambda)\Phi$ is positive.
\end{definition}
The notion of strong positivity was introduced by Nualart and Zakai
in \cite{NZ} and it is related to the positivity improving property
of the Ornstein-Uhlenbeck semigroup. Observe that if $F$ is strongly
positive then it is also positive (according to Definition
\ref{posivity}). Moreover
$F\in\mathcal{L}^p(W,\mu)\subset\mathcal{G}^*$ for some $p>1$ is
positive according to Definition \ref{posivity} if and only if $F$
is non negative in the usual sense, i.e. $\mu(F\geq 0)=1$.
\begin{example}\label{simple example}
Any element of the form $\mathcal{E}(h)$, $h\in H$ is strongly
positive; in fact
\begin{eqnarray*}
\Gamma(\lambda)\mathcal{E}(h)=\mathcal{E}(\lambda h)
\end{eqnarray*}
and therefore $\Gamma(\lambda)\mathcal{E}(h)$ is non negative with
probability one for all $\lambda\geq 1$. In particular, taking $h=0$
we get that non negative constants are strongly positive.\\
It is also straightforward to note that convex combinations of
stochastic exponentials are strongly positive (this is actually a
particular case of Theorem 5.1 in \cite{NZ}).
\end{example}

We now consider a more interesting example

\begin{example}
Let $\{W_{t,x}\}_{0\leq t\leq T, x\in\mathbb{R}}$ be a space-time
white noise and consider the following stochastic partial
differential equation
\begin{eqnarray*}
\partial_t u(t,x)&=&\frac{1}{2}\partial_{xx}u(t,x)+u(t,x)\diamond W_{tx},\quad t\in
]0,T], x\in\mathbb{R}\\
u(0,x)&=&1,\quad x\in\mathbb{R}.
\end{eqnarray*}
Then the unique square integrable mild solution to the previous
Cauchy problem is a
strongly positive element of $\mathcal{G}$.\\
To see this, we proceed at a quite formal level since a rigorous
proof of the previous statement is beyond the scope of this
paper (but can be deduced easily from the results of Hu \cite{H}).\\
For any $\lambda\in\mathbb{R}$ denote with
$\{u^{\lambda}(t,x)\}_{0\leq t\leq T, x\in\mathbb{R}}$ the unique
square integrable mild solution of the Cauchy problem
\begin{eqnarray*}
\partial_t u^{\lambda}(t,x)&=&\frac{1}{2}\partial_{xx}u^{\lambda}(t,x)+\lambda u^{\lambda}(t,x)\diamond W_{tx},\quad t\in
[0,T[, x\in\mathbb{R}\\
u^{\lambda}(0,x)&=&1,\quad x\in\mathbb{R}.
\end{eqnarray*}
It is easy to see that $u^{\lambda}(t,x)=\Gamma(\lambda)u(t,x)$ and
that this last equality implies that for any $\lambda\in\mathbb{R}$,
$\Gamma(\lambda)u(t,x)\in\mathcal{L}^2(W,\mu)$ which is equivalent
to say that $u(t,x)$ belongs to $\mathcal{G}$.\\
Moreover, since it is known (see Bertini and Cancrini \cite{BC})
that for any $\lambda\in\mathbb{R}$ the random variable
$u^{\lambda}(t,x)$ is strictly positive with probability one, from
$u^{\lambda}(t,x)=\Gamma(\lambda)u(t,x)$ we also get the strong
positivity of $u(t,x)$.
\end{example}

The next result is a slight modification of Proposition 5.1 in \cite{NZ}. It becomes a characterization of strong positivity when we
 replace stochastic exponentials with their complex counterparts

\begin{theorem}\label{characterization theorem for strong positivity}
If $\Phi\in\mathcal{G}^*$ is strongly positive then
\begin{eqnarray}\label{strong positivity condition}
\langle\langle \Phi,\varphi\diamond\varphi\rangle\rangle\geq 0\mbox{
for all }\varphi\in\mathcal{E}.
\end{eqnarray}
\end{theorem}
Since $\mathcal{E}$ is dense in $\mathcal{G}$ and the Wick product
is a continuous bilinear operator on that space, condition
(\ref{strong positivity condition}) is equivalent to
\begin{eqnarray*}
\langle\langle \Phi,\varphi\diamond\varphi\rangle\rangle\geq 0\mbox{
for all }\varphi\in\mathcal{G}.
\end{eqnarray*}
In the sequel, strongly positive functions will play the role of
Radon-Nikodym densities of probability measures on
$(W,\mathcal{B}(W))$ with respect to the reference Gaussian measure
$\mu$. We now want to show that measures of this type do not belong
in general to the class of log-concave measures.\\
To this aim consider the finite dimensional abstract Wiener space
given by $H=W=\mathbb{R}^n$, $n\in\mathbb{N}$ and
\begin{eqnarray}\label{finite dimensional}
\mu(A):=\int_{A}(2\pi)^{-\frac{n}{2}}e^{-\frac{|w|^2}{2}}dw,\quad
A\in\mathcal{B}(\mathbb{R}^n),
\end{eqnarray}
where $|\cdot|$ denotes the $n$-th dimensional Euclidean norm.\\
In this framework the class of stochastic exponentials is
represented by the functions
\begin{eqnarray*}
\mathcal{E}(x):=e^{\langle x,w\rangle-\frac{|x|^2}{2}},\quad
w\in\mathbb{R}^n
\end{eqnarray*}
where $x$ is arbitrarily chosen in $\mathbb{R}^n$ and
$\langle\cdot,\cdot\rangle$ denotes, only for this particular
example, the scalar product on $\mathbb{R}^n$.\\
As we mentioned before in Example \ref{simple example} convex
combinations of stochastic exponentials are strongly positive.
Therefore if we define
\begin{eqnarray*}
\varphi(w)=\frac{\mathcal{E}(a)+\mathcal{E}(b)}{2}
\end{eqnarray*}
for some $a,b\in\mathbb{R}^n$ then $\varphi$ is a strongly positive
function. Now, look at $\varphi$ as a Radon-Nicodym density of a
probability measure on $(\mathbb{R}^n,\mathcal{B}(\mathbb{R}^n))$
with respect to the Gaussian measure $\mu$ defined in (\ref{finite
dimensional}); multiplying $\varphi$ by the Gaussian density
appearing in the integral (\ref{finite dimensional}) we will obtain
the density of the above mentioned measure with respect to the
$n$-dimensional Lebesgue measure, i.e.
\begin{eqnarray}\label{non-log-concave}
\varphi(w)\cdot(2\pi)^{-\frac{n}{2}}e^{-\frac{|w|^2}{2}}&=&
\frac{\mathcal{E}(a)+\mathcal{E}(b)}{2}\cdot(2\pi)^{-\frac{n}{2}}e^{-\frac{|w|^2}{2}}\nonumber\\
&=&\frac{e^{\langle a,w\rangle-\frac{|a|^2}{2}}+e^{\langle
b,w\rangle-\frac{|b|^2}{2}}}{2}\cdot(2\pi)^{-\frac{n}{2}}e^{-\frac{|w|^2}{2}}\nonumber\\
&=&(2\pi)^{-\frac{n}{2}}\frac{e^{-\frac{|w-a|^2}{2}}+e^{-\frac{|w-b|^2}{2}}}{2}
\end{eqnarray}
If $a\neq b$ then the function appearing in the last term of the
above chain of equality is not log-concave, i.e. it is not the exponential of a concave function.\\
This shows that probability measures with strongly positive
densities with respect to the standard Gaussian measure do not
belong in general to the class of log-concave measures.

\section{Main results}

We are now ready to state and prove the first main result of the
present paper. In the sequel the symbol $E_{\nu}[\cdot]$ will denote
the expectation with respect to the measure $\nu d\mu$.
\begin{theorem}\label{main theorem}
Let $\nu\in\mathcal{G}_{\sqrt{2}}$ be strongly positive. Then for
any $F\in\mathbb{D}^{1,3}$ we have
\begin{eqnarray}\label{main poincare}
0\leq E_{\nu}[|F|^2]-\langle\langle F\diamond
F,\nu\rangle\rangle\leq E_{\nu}[\Vert DF\Vert_H^2].
\end{eqnarray}
\end{theorem}

\noindent Before proving the theorem above we would like to justify
through a simple example the appearance of the quantity
$\langle\langle F\diamond
F,\nu\rangle\rangle$ in (\ref{main poincare}).\\
Consider the function in (\ref{non-log-concave}) which was shown to
be an admissible candidate for our Poincar\'e inequality but not
fitting in the framework of the Brascamp-Lieb inequality (\ref{BL
introduction}) since it not log-cancave. The function in
(\ref{non-log-concave}) is a convex combination of the two
log-concave densities
\begin{eqnarray*}
p(w-a)=(2\pi)^{-\frac{n}{2}}e^{-\frac{|w-a|^2}{2}}\quad\mbox{ and
}\quad p(w-b)=(2\pi)^{-\frac{n}{2}}e^{-\frac{|w-b|^2}{2}},
\end{eqnarray*}
where we set $p(w):=(2\pi)^{-\frac{n}{2}}e^{-\frac{|w-b|^2}{2}}$.
Therefore the measures
\begin{eqnarray*}
d\nu_a(w):=p(w-a)dw\quad\mbox{ and }\quad d\nu_b(w):=p(w-b)dw
\end{eqnarray*}
satisfy
\begin{eqnarray*}
\int_{\mathbb{R}^n}f^2(w)d\nu_a(w)-\Big(\int_{\mathbb{R}^n}f(w)d\nu_a(w)\Big)^2\leq
\int_{\mathbb{R}^n}|\nabla f(w)|^2d\nu_a(w)
\end{eqnarray*}
and
\begin{eqnarray*}
\int_{\mathbb{R}^n}f^2(w)d\nu_b(w)-\Big(\int_{\mathbb{R}^n}f(w)d\nu_b(w)\Big)^2\leq
\int_{\mathbb{R}^n}|\nabla f(w)|^2d\nu_b(w)
\end{eqnarray*}
respectively. Now sum the two inequalities above and divide by two
to obtain
\begin{eqnarray}\label{convex}
\int_{\mathbb{R}^n}f^2(w)d\nu(w)-\frac{\Big(\int_{\mathbb{R}^n}f(w)d\nu_a(w)\Big)^2+
\Big(\int_{\mathbb{R}^n}f(w)d\nu_b(w)\Big)^2}{2}\leq
\int_{\mathbb{R}^n}|\nabla f(w)|^2d\nu(w)
\end{eqnarray}
where $\nu=\frac{\nu_a+\nu_b}{2}$. Moreover
\begin{eqnarray*}
\Big(\int_{\mathbb{R}^n}f(w)d\nu_a(w)\Big)^2&=&\Big(\int_{\mathbb{R}^n}f(w)p(w-a)dw\Big)^2\\
&=&\Big(\int_{\mathbb{R}^n}f(w)e^{\langle
w,a\rangle-\frac{|a|^2}{2}}p(w)dw\Big)^2\\
&=&\int_{\mathbb{R}^n}(f\diamond f)(w)e^{\langle
w,a\rangle-\frac{|a|^2}{2}}p(w)dw\\
&=&\int_{\mathbb{R}^n}(f\diamond f)(w)d\nu_a(w).
\end{eqnarray*}
Here we used the characterizing property of the Wick product
(\ref{charact}) and the fact that the function $e^{\langle
w,a\rangle-\frac{|a|^2}{2}}$ plays the role of the stochastic
exponential in this framework. If we do the same for the term
$\Big(\int_{\mathbb{R}^n}f(w)d\nu_b(w)\Big)^2$, we can rewrite
inequality (\ref{convex}) as
\begin{eqnarray*}
\int_{\mathbb{R}^n}f^2(w)d\nu(w)-\int_{\mathbb{R}^n}(f\diamond
f)(w)d\nu(w)\leq \int_{\mathbb{R}^n}|\nabla f(w)|^2d\nu(w),
\end{eqnarray*}
which is a particular case of (\ref{main poincare}).\\

\begin{proof}
First of all we observe that since
$F\in\mathbb{D}^{1,3}\subset{L}^2(W,\mu)$ the Wick product
$F\diamond F$ appearing in (\ref{main poincare}) belongs to
$\mathcal{G}_{\frac{1}{\sqrt{2}}}$ (see (\ref{Wick product in G^*}))
and hence the dual pairing $\langle\langle F\diamond
F,\nu\rangle\rangle$ is well defined (according to the assumptions
on $\nu$). \\
We divide the proof in two parts.\\

\noindent\textsc{Second inequality}\\

\noindent Define the map
\begin{eqnarray}\label{T}
T:\mathcal{E}&\to&\mathcal{E}\nonumber\\
F&\mapsto& T(F):=F\diamond F-F^2+\Vert DF\Vert_H^2.
\end{eqnarray}
Since $F\in\mathcal{E}$ we can write
$F=\sum_{k=1}^n\lambda_k\mathcal{E}(h_k)$ for some
$\lambda_1,...,\lambda_n\in\mathbb{R}$ and $h_1,...,h_n\in H$. Now
substitute this expression into (\ref{T}) to obtain
\begin{eqnarray*}
T(F)&=&\sum_{j,k=1}^n\lambda_j\lambda_k\mathcal{E}(h_j)\diamond\mathcal{E}(h_k)-
\sum_{j,k=1}^n\lambda_j\lambda_k\mathcal{E}(h_j)\cdot\mathcal{E}(h_k)\\
&&+\sum_{j,k=1}^n\lambda_j\lambda_k\mathcal{E}(h_j)\cdot\mathcal{E}(h_k)\langle
h_j,h_k\rangle_H\\
&=&\sum_{j,k=1}^n\lambda_j\lambda_k\mathcal{E}(h_j)\diamond\mathcal{E}(h_k)-
\sum_{j,k=1}^n\lambda_j\lambda_k\mathcal{E}(h_j)\diamond\mathcal{E}(h_k)e^{\langle h_j,h_k\rangle_H}\\
&&+\sum_{j,k=1}^n\lambda_j\lambda_k\mathcal{E}(h_j)\diamond\mathcal{E}(h_k)e^{\langle
h_j,h_k\rangle_H}\langle
h_j,h_k\rangle_H\\
&=&\sum_{j,k=1}^n\lambda_j\lambda_k\mathcal{E}(h_j)\diamond\mathcal{E}(h_k)\Big(1-e^{\langle
h_j,h_k\rangle_H}+e^{\langle h_j,h_k\rangle_H}\langle
h_j,h_k\rangle_H\Big).
\end{eqnarray*}
Take the expectation with respect to the measure $\nu d\mu$ of the
first and last terms of the previous chain of equalities to obtain
\begin{eqnarray}\label{inequality}
E_{\nu}[T(F)]&=&\sum_{j,k=1}^n\lambda_j\lambda_kE_{\nu}[\mathcal{E}(h_j)\diamond\mathcal{E}(h_k)]\Big(1-e^{\langle
h_j,h_k\rangle_H}+e^{\langle h_j,h_k\rangle_H}\langle
h_j,h_k\rangle_H\Big)\nonumber\\
&=&\sum_{j,k=1}^n\lambda_j\lambda_ka_{jk}b_{jk},
\end{eqnarray}
where for $j,k\in\{1,...,n\}$ we set
\begin{eqnarray*}
a_{jk}:=E_{\nu}[\mathcal{E}(h_j)\diamond\mathcal{E}(h_k)]
\end{eqnarray*}
and
\begin{eqnarray*}
b_{jk}:=1-e^{\langle h_j,h_k\rangle_H}+e^{\langle
h_j,h_k\rangle_H}\langle h_j,h_k\rangle_H.
\end{eqnarray*}
Observe that the matrix $B=\{b_{jk}\}_{1\leq j,k\leq n}$ is positive
semi-definite; in fact, if in the classical Poincar\'e inequality
\begin{eqnarray*}
E[F^2]-E[F]^2\leq E[\Vert DF\Vert_H^2],
\end{eqnarray*}
we take $F$ to be $\sum_{k=1}^n\lambda_k\mathcal{E}(h_k)$ one gets
\begin{eqnarray*}
\sum_{j,k=1}^n\lambda_j\lambda_k(e^{\langle h_j,h_k\rangle_H}-1)\leq
\sum_{j,k=1}^n\lambda_j\lambda_ke^{\langle h_j,h_k\rangle_H}\langle
h_j,h_k\rangle_H,
\end{eqnarray*}
which corresponds exactly to what we are claiming. On the other
hand, the matrix $A=\{a_{jk}\}_{1\leq j,k\geq n}$ is positive
semi-definite since we are assuming $\nu$ to be strongly
positive (see (\ref{strong positivity condition})).\\
Therefore the matrix $A\circ B:=\{a_{jk}\cdot b_{jk}\}_{1\leq
j,k\leq n}$ (which corresponds to the Hadamard product of the matrix
$A$ with the matrix $B$) is also positive semi-definite (see for
instance Styan \cite{S}), that means
\begin{eqnarray*}
\sum_{j,k=1}^n\lambda_j\lambda_ka_{jk}b_{jk}\geq 0,
\end{eqnarray*}
for any $\lambda_1,...,\lambda_n\in\mathbb{R}$. From
(\ref{inequality}) this corresponds to
\begin{eqnarray*}
E_{\nu}[T(F)]\geq 0\mbox{ for all }F\in\mathcal{E}.
\end{eqnarray*}
From the definition of $T$ this is equivalent to
\begin{eqnarray}\label{approximated inequality}
E_{\nu}[F\diamond F]-E_{\nu}[F^2]+E_{\nu}[\Vert DF\Vert_H^2]\geq 0.
\end{eqnarray}
Therefore we have proved the second inequality in (\ref{main
poincare}) for $F\in\mathcal{E}$ (recall that $\langle\langle
F\diamond F,\nu\rangle\rangle=E_{\nu}[F\diamond F]$ since $F\diamond
F\in\mathcal{L}^2(W,\mu)$ for $F\in\mathcal{E}$).\\
The next step is to extend the validity of (\ref{approximated
inequality}) to the whole $\mathbb{D}^{1,3}$. Since the expectations
in (\ref{approximated inequality}) are taken with respect to the
measure $\nu d\mu$, we need to control the norms in
$\mathcal{L}^2(W,\nu d\mu)$ with the norms in $\mathcal{L}^3(W,\mu)$
(to exploit the density of $\mathcal{E}$ in $\mathbb{D}^{1,3}$).
By the Nelson's hyper-contractivity theorem
and the assumptions on $\nu$ we have that
\begin{eqnarray*}
E[|\nu|^3]^{\frac{1}{3}}&=&E[|\Gamma(1/\sqrt{2})\Gamma(\sqrt{2})\nu|^3]^{\frac{1}{3}}\\
&\leq&E[|\Gamma(\sqrt{2})\nu|^2]^{\frac{1}{2}}\\
&<&+\infty
\end{eqnarray*}
showing that $\nu\in\mathcal{L}^3(W,\mu)$. Hence by the H\"older inequality we deduce
\begin{eqnarray*}
E_{\nu}[|F|^2]^{\frac{1}{2}}&=&E[|F|^2\cdot\nu]^{\frac{1}{2}}\\
&\leq&(E[|F|^3]^{\frac{2}{3}}\cdot E[|\nu|^3]^{\frac{1}{3}})^{\frac{1}{2}}\\
&=&CE[|F|^3]^{\frac{1}{3}}
\end{eqnarray*}
This shows that if $F_n\in\mathcal{E}$ converges to $F$ in
$\mathbb{D}^{1,3}$ then $E_{\nu}[F_n^2]$ and $E_{\nu}[\Vert
DF_n\Vert_H^2]$ converge respectively to $E_{\nu}[F^2]$ and
$E_{\nu}[\Vert DF\Vert_H^2]$. Moreover since  $E_{\nu}[F_n\diamond
F_n]=\langle\langle F_n\diamond F_n,\nu\rangle\rangle$ and
convergence in $\mathbb{D}^{1,3}$ implies convergence in
$\mathcal{G}^*$ (where the Wick product is continuous) we can
conclude that $E_{\nu}[F_n\diamond F_n]$ converges to
$\langle\langle F\diamond F,\nu\rangle\rangle$.\\
We can therefore extend the validity of (\ref{approximated
inequality}) to the whole $\mathbb{D}^{1,3}$.\\

\noindent\textsc{First inequality}\\

\noindent Let $F$ be of the form $p(\delta(h_1),...,\delta(h_n))$
where $p:\mathbb{R}^n\to\mathbb{R}$ is a polynomial and
$h_1,...,h_n\in H$. Functions of this type belong to $\mathcal{G}$
and they are dense in $\mathcal{L}^p(W,\mu)$ and $\mathbb{D}^{k,p}$
for any $p\geq 1$ and $k\in\mathbb{N}$. Then we can write
\begin{eqnarray}\label{formula wick-dot}
F\cdot F=F\diamond F+\sum_{j\geq 1}\frac{1}{j!}\sum_{k\geq 1}\langle
D^jF,e^{(j)}_k\rangle_{H^{\otimes j}}\diamond\langle
D^jF,e^{(j)}_k\rangle_{H^{\otimes j}},
\end{eqnarray}
where for any $j\geq 1$, $\{e^{(j)}_k\}_{k\geq 1}$ is an orthonormal
basis of $H^{\otimes j}$.\\
Formula (\ref{formula wick-dot}) is the dual of (\ref{idea}); it
also appeared for the first time in \cite{NZ} and independently in
\cite{HO} without proof; see \cite{HY} and \cite{L} for the
proof.\\
Observe that the sum over $j$ is finite since $p$ is a polynomial;
moreover rewriting $p(\delta(h_1),...,\delta(h_n))$ as
$\tilde{p}(\delta(\tilde{h}_1),...,\delta(\tilde{h}_k))$ with $k\leq
n$, $\tilde{p}$ a polynomial and $\{\tilde{h}_1,...,\tilde{h}_k\}$
an orthonormal basis for $\mbox{span}\{h_1,...,h_n\}$ one can make
also the sum over $k$ to be finite (choosing proper basis
$\{e^{(j)}_k\}_{k\geq 1}$). Hence
\begin{eqnarray*}
E_{\nu}[F^2]&=&E_{\nu}[F\diamond F]+\sum_{j\geq
1}\frac{1}{j!}\sum_{k\geq 1}E_{\nu}[(\langle
D^jF,e^{(j)}_k\rangle_{H^{\otimes j}})^{\diamond 2}]\\
&\geq& E_{\nu}[F\diamond F],
\end{eqnarray*}
since all the terms inside the last sum are non negative (due to the
strong positivity of $\nu$ (see \ref{strong positivity condition})).
This shows that
\begin{eqnarray*}
E_{\nu}[F^2]-E_{\nu}[F\diamond F]\geq 0
\end{eqnarray*}
or equivalently
\begin{eqnarray*}
E_{\nu}[F^2]-\langle\langle F\diamond F,\nu\rangle\rangle\geq 0.
\end{eqnarray*}
for any $F$ of polynomial type. The density of these functions in
$\mathbb{D}^{1,3}$, the continuity of the Wick product in
$\mathcal{G}^*$ and the assumptions on $\nu$ complete the proof.
\end{proof}

\begin{remark}
Theorem \ref{main theorem} is a generalization of the classical
Poincar\'e inequality on an abstract Wiener space, i.e.
\begin{eqnarray}\label{classical poincare}
E[F^2]-E[F]^2\leq E[\Vert DF\Vert_H^2],\quad F\in\mathbb{D}^{1,2}.
\end{eqnarray}
In fact, choosing $\nu=1$ in (\ref{main poincare}), which clearly is
strongly positive and belongs to $\mathcal{G}_{\sqrt{2}}$, one gets
precisely (\ref{classical poincare}) (recall equality (\ref{Wick
expectation}) and the fact for $\nu=1$ we can replace $\mathbb{D}^{1,3}$ with $\mathbb{D}^{1,2}$). Moreover, the family of probability measures for
which our Poincar\'e inequality is proved, namely measures of the
type $\nu d\mu$ with $\nu$ strongly positive, is in general not
included in the class of log-concave measures, which is the class of
probability measure considered by Brascamp and Lieb \cite{BL} (see
the example at the end of Section 3) and Feyel and \"Ust\"unel
\cite{FU} (in the abstract Wiener space setting).
\end{remark}

\begin{remark}
If $F\in\mathbb{D}^{1,3}$ and $\nu$ is a strongly positive element
of $\mathcal{G}_{\sqrt{2}}$ then we have
\begin{eqnarray*}
(E_{\nu}[F])^2\leq \langle\langle F\diamond F,\nu\rangle\rangle\leq
E_{\nu}[F^2].
\end{eqnarray*}
The second inequality is part of Theorem \ref{main theorem}. The
first inequality is easily obtained as follows:
\begin{eqnarray*}
0&\leq& E_{\nu}[(\varphi-E_{\nu}[\varphi])\diamond
(\varphi-E_{\nu}[\varphi])]\\
&=&E_{\nu}[\varphi\diamond\varphi]-(E_{\nu}[\varphi])^2.
\end{eqnarray*}
Here $\varphi\in\mathcal{E}$ and we utilized the property that
$\varphi\diamond\psi=\varphi\cdot\psi$ if $\varphi$ or $\psi$ is
constant. Therefore
\begin{eqnarray*}
E_{\nu}[F^2]-\langle\langle F\diamond F,\nu\rangle\rangle&\leq&
E_{\nu}[F^2]-(E_{\nu}[F])^2\\
&=&Var_{\nu}(F)
\end{eqnarray*}
where $Var_{\nu}(F)$ stands for the variance of $F$ under the
measure $\nu d\mu$.
\end{remark}

\begin{remark}
Theorem \ref{main theorem} has been proved for functions in $\mathbb{D}^{1,3}$; the natural space would be $\mathbb{D}_{\nu}^{1,2}$ 
where we mean the Sobolev space with respect to the measure $\nu d\mu$. This space is well defined through the closability of the gradient operator with respect to the norm
in $\mathcal{L}^2(W,\nu d\mu)$ which is however not known in the general case. This is why for instance the Poincar\'e inequality for log-concave measures on abstract Wiener spaces obtained in \cite{FU} is proved only for cilindrical smooth functions.
\end{remark}

\subsection{A refinement on the classical Wiener space}

In this section we are going to generalize Theorem \ref{main
theorem} in the spirit of the Poincar\'e type inequality obtained by
Houdr\'e and Kagan in \cite{HK} for the finite dimensional case and
later on by Houdr\'e and P\'erez-Abreu in \cite{HPA} for the
classical Wiener space. More precisely, in the paper \cite{HPA} it
is proved that on the classical Wiener space the following
inequality holds for any $F\in\mathbb{D}^{2k,2}$ with
$k\in\mathbb{N}$:
\begin{eqnarray}\label{HK}
\sum_{l=1}^{2k}\frac{(-1)^{l+1}}{l!} E[\Vert D^lF\Vert_{H^{\otimes
l}}^2]\leq E[|F|^2]-E[F]^2\leq
\sum_{l=1}^{2k-1}\frac{(-1)^{l+1}}{l!} E[\Vert D^lF\Vert_{H^{\otimes
l}}^2].
\end{eqnarray}
Its proof relies on iterations of the Clark-Ocone formula (see
\cite{Nualart}) and we do not know whether inequality (\ref{HK}) is
valid in the abstract Wiener space setting; since the proof of our
next theorem will rely on the validity of inequality (\ref{HK}) we
are forced to work in the framework of the classical Wiener space.\\
We will generalize the second inequality in (\ref{HK}) replacing the
expectations $E[\cdot]$ with $E_{\nu}[\cdot]$, where $\nu$ is a
strongly positive element of $\mathcal{G}_{\sqrt{2}}$, and the term
$E[F]^2$ with $\langle\langle F\diamond F,\nu\rangle\rangle$. As
before, these two last expressions coincide for $\nu=1$.
\begin{theorem}\label{main theorem 2}
Let $(H,W,\mu)$ be the classical Wiener space and let
$\nu\in\mathcal{G}_{\sqrt{2}}$ be strongly positive. Then for any
$k\in\mathbb{N}$ and $F\in\mathbb{D}^{2k-1,3}$ we have
\begin{eqnarray}\label{main poincare 2}
0\leq E_{\nu}[|F|^2]-\langle\langle F\diamond
F,\nu\rangle\rangle\leq \sum_{l=1}^{2k-1}\frac{(-1)^{l+1}}{l!}
E_{\nu}[\Vert D^lF\Vert_{H^{\otimes l}}^2].
\end{eqnarray}
\end{theorem}

\begin{proof}
The proof is similar to one of Theorem \ref{main theorem}; one has
to start with the map
\begin{eqnarray}\label{T2}
T:\mathcal{E}&\to&\mathcal{E}\nonumber\\
F&\mapsto& T(F):=F\diamond
F-F^2+\sum_{l=1}^{2k-1}\frac{(-1)^{l+1}}{l!}\Vert
D^lF\Vert_{H^{\otimes l}}^2
\end{eqnarray}
and observe that for $F=\sum_{k=1}^n\lambda_k\mathcal{E}(h_k)$ one
gets
\begin{eqnarray*}
T(F)&=&\sum_{j,k=1}^n\lambda_j\lambda_k\mathcal{E}(h_j)\diamond\mathcal{E}(h_k)\Big(1-e^{\langle
h_j,h_k\rangle_H}+\sum_{l=1}^{2k-1}\frac{(-1)^{l+1}}{l!}e^{\langle
h_j,h_k\rangle_H}\langle h_j,h_k\rangle^l_H\Big).
\end{eqnarray*}
Now the matrix $A=\{a_{i,j}\}_{1\leq i,j\leq n}$ defined by
\begin{eqnarray*}
a_{ij}:=1-e^{\langle
h_j,h_k\rangle_H}+\sum_{l=1}^{2k-1}\frac{(-1)^{l+1}}{l!}e^{\langle
h_j,h_k\rangle_H}\langle h_j,h_k\rangle^l_H
\end{eqnarray*}
is positive semi-definite by the second inequality in (\ref{HK})
applied to $F=\sum_{k=1}^n\lambda_k\mathcal{E}(h_k)$. The proof then
follows as before.
\end{proof}

\end{document}